%% file: main_ifac.tex
\documentclass{ifacconf}

\input{preamble}

\begin{document}
\begin{frontmatter}

\title{Anderson Acceleration for Linearly Converging SQP-Type Methods \thanksref{footnoteinfo}}

\thanks[footnoteinfo]{
This research was supported by DFG via projects 504452366 (SPP 2364) and 525018088, and by BMWK via 03EN3054B.
}

\author[First,Email]{Jonathan Frey\thanksref{footnote_equal}}
\author[First]{David Kiessling\thanksref{footnote_equal}}
\author[First]{Katrin Baumgärtner}
\author[First,Second]{Moritz Diehl} 

\address[First]{Department of Microsystems Engineering, University of Freiburg.}
\address[Second]{Department of Mathematics, University of Freiburg.}
\address[Email]{email: \texttt{jonathan.frey@imtek.uni-freiburg.de}}

\thanks[footnote_equal]{
The authors contributed equally.
}
\begin{abstract}
    Although Anderson acceleration (AA) is known to speed up fixed-point iterations, it is rarely applied in constrained optimization, in particular sequential quadratic programming (SQP).
    We show that the local convergence behavior of a general family of (inexact) SQP-type methods can benefit from AA and introduce a simple heuristic to alleviate slower convergence farther from the solution.
    The method is implemented in the software framework \texttt{acados}.
    Numerical examples from optimal control illustrate consistent improvements in convergence of different SQP-type methods.
\end{abstract}


\end{frontmatter}
\input{contents}

\bibliography{squeeze_syscop}             %

\end{document}

%% file: preamble.tex
\usepackage{amsmath,amssymb,amsfonts}
\renewenvironment{pf}%
  {\par\addvspace{3pt}\noindent{\itshape Proof.}\enspace\ignorespaces}%
  {\par\addvspace{3pt}}
\newenvironment{proof}{\begin{pf}}{\end{pf}}
\usepackage{graphicx}
\usepackage{textcomp}
\usepackage[utf8]{inputenc}
\usepackage{float}

\usepackage{booktabs}
\usepackage{url}  %
\usepackage{natbib}
\usepackage{courier}
\usepackage{pgf}
\usepackage{optidef}

\usepackage{tikz,pgfplots}
\usepackage{epstopdf}%
\usepackage[normalem]{ulem}
\usepackage{siunitx}
\usepackage{epigraph}

\usepackage{xcolor}
\usepackage{bm}
\usepackage{lscape}

\usepackage[shortlabels]{enumitem}

\usepackage{pifont}

\pgfplotsset{compat=1.15}

\usepackage{comment}

\usepackage{inconsolata}
\usetikzlibrary{arrows.meta}

\newtheorem{lemma}{Lemma}
\newtheorem{theorem}{Theorem}
\newtheorem{assumption}{Assumption}
\newtheorem{proposition}{Proposition}
\newtheorem{remark}{Remark}

\newcommand{\R}{\mathbb{R}}

\newcommand{\ind}[1]{{_{\mathrm{#1}}}}

\newcommand{\aatresh}{\delta\ind{AA}}
\newcommand{\controp}{\varphi}

\DeclarePairedDelimiter\abs{\lvert}{\rvert}%
\DeclarePairedDelimiter\norm{\lVert}{\rVert}%

\makeatletter
\let\oldabs\abs
\def\abs{\@ifstar{\oldabs}{\oldabs*}}
\let\oldnorm\norm
\def\norm{\@ifstar{\oldnorm}{\oldnorm*}}
\makeatother

\newcommand{\T}{^{\top}}

\newcommand{\bigo}{{\mathcal O}}
\DeclareMathSymbol{\sm}{\mathbin}{AMSa}{"39}

%% file: contents.tex
\section{INTRODUCTION}
Anderson acceleration (AA) is a general technique which allows speeding up convergence in linearly converging fixed-point iterations.
While originally proposed in the context of integral equations~\citep{Anderson1965} the method has gained general popularity in various fields, in particular for large-scale systems arising from PDE discretizations \citep{Delaisse2023}.
However, AA is rarely used in optimization solvers, especially constrained and nonlinear ones.
The present paper discusses the applicability of AA in the context of a general class of (inexact) SQP-type methods for nonlinear optimization and studies its benefits on practical optimal control examples.

In this paper, we consider nonlinear programming (NLP) formulations of the form
\begin{mini!}|s|
{\scriptstyle{\substack{v}\in\R^{n_v}}}
{f(v)}{\label{eq:NLP}}{}
\addConstraint{g(v)}{= 0}
\addConstraint{h(v)}{\leq 0,}
\end{mini!}
where $f\colon \mathbb{R}^{n_v} \rightarrow \mathbb{R}$, $g\colon \R^{n_v} \rightarrow \R^{n_{g}}$, and $h\colon \R^{n_v} \rightarrow \R^{n_{h}}$ are assumed to be twice continuously differentiable.
Optimal control problems (OCP) can be phrased in this standard form and are repeatedly solved in model predictive control (MPC).
Especially in this context, sequential quadratic programming (SQP) methods have been shown to work particularly well, as they have beneficial warm-starting strategies and allow giving fast feedback~\citep{Gros2020b}.

SQP employing an exact Hessian exhibits quadratic convergence behavior.
The method is rarely used in practical SQP-type algorithms, as the exact Hessian of the Lagrangian is in general indefinite, rendering it unsuitable for most QP solvers.
In contrast, the Gauss-Newton Hessian approximation and related convexity exploiting approaches are widely used in SQP-type methods, as they provide an inherently positive-semidefinite Hessian approximation which is cheap to compute and result in desirable convergence properties \citep{Verschueren2016, Messerer2021a, Baumgaertner2022}.
Additionally, even if the exact Hessian is computed, it is often modified, e.g. by adding a positive constant onto its diagonal, which is referred to as a \textit{Levenberg-Marquardt} or \textit{proximal} term \citep{Nocedal2006}, or by employing a more elaborate regularization technique~\citep{Verschueren2017}.
Inexact SQP-type methods have recently gained attraction, as they allow to significantly reduce the computational complexity at the cost of converging to a suboptimal point \citep{Bock2007, Zanelli2016, Baumgaertner2019}.

The open-source software package \texttt{acados} implements efficient SQP-type algorithms for embedded optimal control \citep{Verschueren2021}.
It is particularly efficient for MPC by rigorously exploiting the OCP structure, being written in \texttt{C} and relying on the high performance linear algebra package \texttt{BLASFEO} \citep{Frison2018}.

In this paper, we show that a variety of SQP-type methods used within the context of numerical optimal control profit from Anderson acceleration yielding improved local convergence.
Since the convergence results hold only locally, AA should be employed within a neighborhood of the solution.
To this end, we suggest a simplified variant of the heuristic proposed by~\cite{Pollock2021}, which we found to work particularly well in practice.
This technique is implemented for the \texttt{acados}~SQP-type OCP solvers, and its effectiveness is shown on three representatives of the class of SQP-type methods.

\subsection{Related work}
\vspace{-2mm}
As there are numerous works on AA for a wide variety of applications, we focus here on AA in the context of numerical optimization.
In the context of first-order methods for QPs and semidefinite programming formulations, \cite{Garstka2022} combined AA with a safe-guarding technique, and \cite{Pereira2025} proposed a variant of AA which mitigates numerical issues when subsequent residuals are nearly collinear.
In the context of Sequential Linear Programming, AA has been shown to be effective in \cite{Kiessling2023}.
While AA was already proposed in 1965, its convergence properties have been studied recently.
For example, \cite{Zhang2020a} derived a global convergence result for safeguarded AA for nonsmooth fixed-point iterations.
In~\cite{Evans2020} and \cite{Pollock2021}, AA is studied in a general setting.
The convergence results in this paper are based on the ones given in \cite{Evans2020}.

\subsection{Outline}
\vspace{-2mm}

The remainder of this paper is structured as follows:
Section~\ref{sec:anderson} introduces AA and a local convergence result.
Section~\ref{sec:sqp_type} gives an overview on SQP-type methods and how the AA convergence result applies to them.
Section~\ref{sec:implementation} presents practical considerations for AA and their implementation, which are evaluated in Section~\ref{sec:examples}.
The paper is concluded in Section~\ref{sec:conclusion}.

\vspace{-1mm}

\section{Anderson Acceleration}\label{sec:anderson}
\vspace{-1.5mm}

A classic fixed-point iteration algorithm starts with a given initial guess $z_0\in\R^n$ and iterates by setting $z_{k+1} = \controp(z_k)$ for $k\ge 0$.
If the operator $\controp$ is a contraction, the iterates converge to a fixed point $z^\star$ satisfying
\begin{align*}
    \controp(z^\star) = z^\star.
\end{align*}

\subsection{General Anderson acceleration algorithm}
\vspace{-2mm}
\newcommand{\aad}{r}
In order to define and analyze the Anderson acceleration algorithm, we introduce the $k$-th differences between iterates as well as the $k$-th residuals as
\begin{subequations}
    \label{eq:residualsAndError}
    \begin{align}
        e_k&:= z_k-z_{k-1},\\
        \aad_k&:= \controp\left(z_{k}\right)-z_{k}.
    \end{align}
\end{subequations}
\vspace{-1mm}

The Anderson-accelerated version of the iteration $z_{k+1} = \controp(z_k)$ with depth $m$, in the following denoted by AA$(m)$, and damping factors $\beta_k\in(0, 1]$ performs the following steps at iteration $k$:
\vspace{-1mm}
\begin{enumerate}[(S1)]
    \item Set current depth $m_k = \min\{k,\,m\}$
    \item Compute $\aad_{k}=\controp\left(z_k\right)-z_k$.
    \item Compute coefficients $\alpha^{k} \in \R^{m_k+1}$ by solving
        \begin{mini!}|s|
        {{\substack{\alpha_{k-m_k}^{k}, \ldots, {\alpha_{k}^{k}}}}}
        {\left\|\sum_{j=k-m_k}^k \alpha_j^{k} \aad_{j}\right\|_2}{\label{eq:AndersonAlgMinimization}}{}
        \addConstraint{\sum_{j=k-m_k}^k \alpha_j^{k}}{=1.}
        \end{mini!}
    \item Compute the next iterate using
        \begin{align}
        \label{eq:AndersonAlgUpdateStep}
        z_{k+1}=\sum_{j=k-m_k}^k \alpha_j^{k} z_j+\beta_k \sum_{j=k-m_k}^k \alpha_j^{k} \aad_{j}.
        \end{align}
\end{enumerate}

Let us define the following averages given by the solution $\alpha^{k}=\left\{\alpha_j^{k}\right\}_{j=k-m_k}^k$ to the optimization problem \eqref{eq:AndersonAlgMinimization}:
\begin{align}
\label{eq:AndersonAverages}
z_k^\alpha=\sum_{j=k-m_k}^k \alpha_j^{k} z_j, \quad \aad_{k}^\alpha=\sum_{j=k-m_k}^k \alpha_j^{k} \aad_{j}.
\end{align}

Then the update \eqref{eq:AndersonAlgUpdateStep} can be written as %
\begin{align}
\label{eq:AndersonAveragedUpdate}
z_{k+1}=z_k^\alpha+\beta_k \aad_{k}^\alpha.
\end{align}
The iteration-$k$ gain $\theta_k$ quantifies the success of the acceleration and is defined by
\begin{align}
\label{eq:OptimizationGain}
\left\|\aad_k^\alpha\right\|=\theta_k\left\|\aad_k\right\|.
\end{align}
Since $\alpha_k^{k} = 1, \alpha_j^{k} = 0 $ for $j < k$ is an admissible point of~\eqref{eq:AndersonAlgMinimization}, it follows that $\theta_k \leq 1$.

\subsection{Formulation as a linear least-squares problem}
The following equivalent view on~\eqref{eq:AndersonAlgMinimization} is particularly useful for practical implementations.
Define the matrices $E_k$ and $F_k$ formed by the respective differences between consecutive iterates and residuals by
\begin{align*}
E_k &:=\left[e_k,\, e_{k\sm 1},\, \dots,\, e_{k\sm m_k+1}\right], \\
F_k &:=\left[\left(\aad_{k}\sm \aad_{k-1}\right), \ldots,\left(\aad_{k\sm m_k+1}\sm \aad_{k\sm m_k}\right)\right].
\end{align*}
By defining
\vspace{-1.5mm}
\begin{align}
    \gamma_j^{k}=\sum_{n=k-m_k}^{j-1} \alpha_n^{k}, %
\end{align}
and $\gamma^{k}=\left(\gamma_k^{k}, \gamma_{k-1}^{k}, \ldots, \gamma_{k-m_k+1}^{k}\right)^{\top}$,
problem~\eqref{eq:AndersonAlgMinimization} is equivalent to the unconstrained problem
\begin{align}
    \label{eq:unconstrained_aa}
    \gamma^{k}=\operatorname{argmin}_{\gamma \in \mathbb{R}^m}\left\|\aad_{k}\sm F_k \gamma\right\|_2.
\end{align}
The averages $z_k^\alpha$ and $\aad_{k}^\alpha$ used in the update \eqref{eq:AndersonAverages} and the transformation between the two sets of optimization coefficients are related by
\begin{subequations}
    \label{eq:AndersonAveragesGamma}
    \begin{align}
    z_k^\alpha     & =z_k-E_k \gamma^{k},\\
    \aad_{k}^\alpha & =\aad_{k}-F_k \gamma^{k}.
\end{align}
\end{subequations}

\subsection{Special case: AA$(1)$ without damping}
In large parts of this paper, we limit the discussion to the case $m=1$ and $\beta_k=1$.
This choice is motivated by initial experiments showing that this special case performs extremely well in the context of SQP-type methods.
In addition, it is easy to implement and does not add tuning parameters.

In the special case of depth $m=1$, the linear least-squares problem \eqref{eq:unconstrained_aa} has the following closed-form solution
\begin{align}
    \gamma^{k} = \frac{\aad_{k}\T (\aad_{k} - \aad_{k-1})}
    {\norm{\aad_{k} - \aad_{k-1}}_2^2},
\end{align}
and the update can be written as
\begin{align}
    z_{k+1} =
    z_k - \gamma^{k} (z_k-z_{k-1}) + \gamma^{k} (\aad_{k} - \aad_{k-1}).
\end{align}
This makes it particularly easy to implement and cheap to compute, as only vector operations are required.

In this special case, the connection between AA and Broyden's update can be easily shown.
In the more general case, this connection has been investigated e.g. by \cite{Walker2011}.
In particular, Broyden's method for solving $\controp(z_k) - z_k = 0$ performs the updates
\begin{align}
    z_{k+1} = z_k - B_k (\controp(z_k) - z_k),
\end{align}
with the inverse Jacobian update
\begin{align}
B_k = B_{k-1}
    + \frac{(\Delta z_k - B_{k-1}\Delta \aad_{k})\, \Delta \aad_{k}^\top}
    {\Delta \aad_{k}^\top \Delta \aad_{k}},
\end{align}
where $\Delta z_k = z_k - z_{k-1}$ and $\Delta \aad_{k} =  \aad_{k} - \aad_{k-1}$.
As we will show now, AA(1) without damping performs the very same update for a special choice of $B_{k-1}$.
Using \eqref{eq:AndersonAveragesGamma}, we can rewrite AA(1) without damping as
\begin{align}
z_{k+1}
&= z_k -e_k \gamma^{k} + \aad_{k} - (\aad_{k} - \aad_{k-1})\gamma^{k},
\end{align}
which simplifies to
\begin{align}
z_{k+1} = (1 - \gamma^{k})\controp(z_k) + \gamma^{k} \controp(z_{k-1}).
\end{align}
With $\controp(z_k) = \aad_{k} + z_k$, we obtain
\begin{align}
z_{k+1}
= z_k + \aad_{k} - \gamma^{k} (\aad_{k} - \aad_{k-1}  + z_k - z_{k-1}).
\end{align}
Using the definition of $\gamma^{k}$, yields
\begin{align}
z_{k+1}
&= z_k + \aad_{k} - \frac{(\Delta \aad_{k})\T \aad_{k}} {\norm{\Delta \aad_{k} }_2^2} (\Delta \aad_{k}  + \Delta z_k) \\
&= z_k - \left(-I + \frac{(\Delta \aad_{k}  + \Delta z_k) (\Delta \aad_{k})\T}
    {\norm{\Delta \aad_{k}}_2^2}  \right) \aad_{k}.
\end{align}
The term within the parentheses coincides with Broyden’s rank-one inverse update starting from $B_{k-1} = -I$.
This shows that, with $B_{k-1} = -I$, Broyden's method and AA(1) without damping yield identical iterates since AA(1) without damping applies a residual-minimizing correction in the same rank-one direction as the secant update in Broyden's method.

\subsection{Convergence results}
This section states the convergence result presented by \cite{Evans2020} specialized for the case $m=1$ and $\beta=1$ and a contractive operator $\controp$ on $\R^n$.

\begin{assumption} %
\label{ass:Evans_31_32}
Let $\controp:\R^n \to \R^n$ have a fixed point $z^\star \in \R^n$.
Let $\controp$ be twice continuously differentiable in a neighborhood $\mathcal{N}$ of $z^\star$ and assume there is a constant $\kappa < 1$ and a compact set $ \mathcal{M} \subset \mathcal{N}$ such that
\begin{align}
    \norm{(\controp(y) - \controp(x))} \le \kappa \norm{x-y}  \;\; \text{for all} \; x, y\in \mathcal{M}.
\end{align}
\end{assumption}

\begin{theorem} \label{thm:AA}
(Convergence of the residual with depth $m = 1$ and $\beta=1$).
Given Assumption~\ref{ass:Evans_31_32},
if the coefficients $\alpha^{k}_k,\alpha^{k-1}_{k-1}$ remain bounded and bounded away from zero, the following bound holds for the residual $\aad_{k}$ for AA(1):
\begin{align}
    \norm{\aad_{k}} &\le \theta_{k-1} \kappa \norm{\aad_{k-1}} + \bigo(\norm{\aad_{k-1}}^2 ) + \bigo(\norm{\aad_{k-2}}^2 ).
\end{align}
\end{theorem}

\begin{proof}
    The theorem is a special case of Theorem 4.1 in \cite{Evans2020}.
    Assumption~\ref{ass:Evans_31_32} is a simplified version of the ones in \cite{Evans2020} which assumes a local setting and a compact set $\mathcal{M}$, such that Lipschitz continuity is implied.
\end{proof}

Since $\theta_k \le 1$, AA(1) accelerates convergence and the improvement is characterized by $\theta_k$.

\section{SQP-type methods}
\label{sec:sqp_type}

In the following, we consider general SQP-type methods that aim at finding an (approximate) solution to \eqref{eq:NLP} and  analyze their local convergence behavior.
Furthermore, we briefly discuss methods from the literature that fit into this general class of SQP-type methods.
As our analysis will show, the considered methods converge at a locally linear rate and thus benefit from Anderson acceleration.
\subsection{Local convergence behavior of SQP-type methods}

The Lagrange function associated with \eqref{eq:NLP} is
\begin{align}
\mathcal{L}(v, \lambda, \mu) = f(v) + \lambda\T g(v) + \mu\T h(v),
\end{align}
where we introduced multipliers $\lambda \in\R^{n_g}$ associated with the equality constraints $g(v)=0$ and $\mu \in\R^{n_h}$ associated with the inequalities $h(v) \leq 0$.
We consider a primal-dual iterate $z_k = (v_k, \lambda_k, \mu_k)$ and define the QP subproblem
\begin{mini!}|s|
  {\scriptstyle{v}}{\tfrac{1}{2}(v - v_k)^{\top} W(z_k) (v - v_k) + q(z_k)\T (v - v_k)}{\label{eq:local-QP}}{}
  \addConstraint{g(v_k) + G(v_k)\T (v-v_k)}{= 0}
  \addConstraint{h(v_k) + H(v_k)\T (v-v_k)}{\leq 0,}
\end{mini!}
where $W(z) \succeq 0$ is a positive semidefinite approximation of the Hessian of the Lagrangian, $W(z) \approx \nabla^2_{\! v}\mathcal{L}(v, \lambda, \mu)$.
Furthermore, we allow all first-order derivatives to be approximate, i.e.,
\begin{align} \label{eq:gradients}
q(z) \approx \nabla f(v), ~~
G(v) \approx \nabla g(v), ~~
H(v) \approx \nabla h(v).
\end{align}
We point out that the gradient $q(z)$ is allowed to depend on the dual iterates in order to cover adjoint gradient corrections as used in e.g. \cite{Bock2007, Frey2024a}.

We define the (inexact) SQP-type iteration as
\begin{align} \label{eq:iteration-map}
z_{k+1}
=
\pi(z_k)
\end{align}
where $\pi(z_k)$ denotes the primal-dual solution operator associated with the QP \eqref{eq:local-QP}, i.e., $z_{k+1}$ satisfies the KKT conditions associated with \eqref{eq:local-QP} given by:
\begin{subequations}
    \begin{align} \label{eq:KKT}
        W_k\big(v_{k+1}-v_k\big) + q_k + G_k\T \lambda_{k+1} + H_k\T \mu_{k+1} &= 0,\\
        g_k + G_k\T \big(v_{k+1}-v_k\big) &= 0,\\
        h_k + H_k\T \big(v_{k+1}-v_k\big) &\leq 0,\\
        \mu_{k+1} &\geq 0, \\
        (\mu_{k+1})_i \!\left(h_k + H_k\T \big(v_{k+1}-v_k\big)\right)_i &= 0,
    \end{align}
\end{subequations}
where $i= 1, \ldots, n_h$ and the index $k$ denotes evaluation of the problem functions at $v_k$ or $z_k$.

\begin{proposition}[Perturbed problem \citep{Bock2007}] \label{prop:perturbed-nlp}
~\\
Suppose the iteration $z_{k+1} = \pi(z_k)$ defined in \eqref{eq:iteration-map} converges to $z^\star$.
If $z^\star$ satisfies LICQ, then $z^\star = (v^\star, \lambda^\star, \mu^\star)$ is a KKT point of the following NLP:
\begin{mini!}|s|
{\scriptstyle{\substack{v}\in\R^{n_v}}}
{f(v) + p\T v}{\label{eq:NLP-perturbed}}{}
\addConstraint{g(v)}{= 0}
\addConstraint{h(v)}{\leq 0,}
\end{mini!}
with
$p = e^f_\star + E^G_\star \lambda^\star + E^H_\star \mu^\star$
where $e^f_\star = \nabla f(v^\star) - q_\star$ and
\begin{align}
E^G_\star = \nabla g(v^\star) - G_\star, ~~E^H_\star = \nabla h(v^\star) - H_\star,
\end{align}
where $q_\star = q(z^\star)$, $G_\star = G(v^\star)$, $H_\star = H(v^\star)$ are the gradient approximations at $v^\star$.
\end{proposition}

\begin{remark}
If $p$ in Proposition~\ref{prop:perturbed-nlp} is zero, then the equilibrium point $z^\star$ is a KKT point of the original NLP \eqref{eq:NLP}.
In this case, we refer to the method as an \textit{exact} SQP-type method.
In particular, this is the case if all first-order derivatives in \eqref{eq:gradients} are evaluated exactly.
If $p\neq 0$, we say the method is an \textit{inexact} SQP-type method.
\end{remark}

Next, we analyze the local convergence behavior of the (inexact) SQP-type method.
To this end, let us characterize $\pi(z)$ in a neighborhood of $z^\star$.

\begin{assumption}\label{ass:regularity_nlp}
Let $z^\star$ be a fixed point of the SQP-type iteration defined in \eqref{eq:iteration-map}.
Suppose second-order sufficient conditions of optimality (SOSC), the linear independence constraint qualification (LICQ) and strict complementarity hold at $z^\star$ with respect to \eqref{eq:local-QP} with $z_k = z^\star$.
For the formal definitions of these assumptions, we refer to \cite{Nocedal2006}.
\end{assumption}

\begin{lemma} \label{lem:root-finding}
Suppose Assumption~\ref{ass:regularity_nlp} holds, in a neighborhood of $z^\star$, the iteration map $z_{k+1} = \pi(z_k)$ is given by the solution to the parametric root-finding problem $R(z; z_k) = 0$ with the residual
\begin{align} \label{eq:residual}
R(z; \bar z) =
\begin{bmatrix}
W(\bar z)(v - \bar v) + q(\bar z) + \tilde G(\bar v) \tilde\lambda \\
\tilde g(\bar v) + \tilde G(\bar v)\T (v - \bar v)
\end{bmatrix},
\end{align}
where $\tilde g(v_k) = (g(v_k), h_\mathcal{A}(v_k))$ and $\tilde G(\bar v) = (G(\bar v), H_\mathcal{A}(\bar v))$ summarize the strictly active constraints associated with multipliers $\tilde \lambda = (\lambda, \mu_\mathcal{A})$.
Furthermore, the map $\pi$ is differentiable in a neighborhood of $z^\star$.

\end{lemma}

\begin{proof}
LICQ and strict complementarity imply that the active set is stable in a neighborhood of $z^\star$, which implies that the inactive inequalities can be omitted and the active inequalities can be treated as equalities such that the KKT conditions \eqref{eq:KKT} reduce to the residual function \eqref{eq:residual}.
SOSC allows us to employ the implicit function theorem which implies differentiability of $\pi$ in a neighborhood of~$z^\star$.
\end{proof}

\begin{theorem}[Linear local contraction rate] \label{thm:rate-spectral-radius}
Let $z^\star$ be a fixed point of the SQP-type iteration defined in \eqref{eq:iteration-map} satisfying Assumption~\ref{ass:regularity_nlp}.

Consider the iteration matrix $A_\star = K_\star^{-1} L_\star$ where
\begin{align}
K_\star &=
\frac{\partial R}{\partial z} (z^\star; z^\star)
=
\begin{bmatrix}
W_\star & \tilde G_\star \\
\tilde G_\star\T & 0
\end{bmatrix}\!, \\
L_\star
&=
\frac{\partial R}{\partial \bar z} (z^\star; z^\star)
=
\begin{bmatrix}
E_\star& D_\star \\
C_\star & 0 \\
\end{bmatrix}\!,
\end{align}
with $W_\star = W(z^\star)$, $\tilde G_\star = \tilde G(v^\star)$, and
\begin{align}
E_\star &= \frac{\partial}{\partial w} q(z_\star) + \frac{\partial}{\partial w} (\tilde G(v_\star) \tilde\lambda_\star) -  W(v^\star), \\
C_\star &= \nabla \tilde{g}(v_\star)\T - \tilde G(v_\star)\T, \\
D_\star &=\frac{\partial}{\partial \tilde{\lambda}_\star} q(z^\star).
\end{align}

The local convergence behavior of the SQP-type iterates is determined by the spectral radius $\kappa = \rho(A_\star)$:

If $0 < \kappa < 1$, the iterates converge Q-linearly with asymptotic rate $\kappa$.
If $ \kappa = 0$, the iterates converge superlinearly to $z^\star$.
If $\kappa = 1$, we cannot decide about convergence by considering only the spectral radius, and if $\kappa > 1$, then $z^\star$ is an unstable stationary point.
\end{theorem}

\begin{proof}
Lemma~\ref{lem:root-finding} implies that in a neighborhood of $z^\star$, the iteration $z_{k+1} = \pi(z_k)$ is given by the solution of the root-finding problem $R(z; z_{k}) = 0$.
The implicit function theorem thus implies differentiability of $\pi$ in a neighborhood of $z^\star$ with its derivative at  $z^\star$ given by:
\begin{align}
\frac{ \partial\pi}{\partial z}(z^\star) = -\frac{\partial R}{\partial z}^{-1} \!(z^\star; z^\star)  \frac{\partial R}{\partial \bar z}(z^\star; z^\star).
\end{align}
A first-order Taylor expansion at $z^\star$ yields
\begin{align}
z_k - z_\star = -A_\star(z_k - z_\star) + \mathcal{O}\left(z_k - z_\star\right).
\end{align}
By a standard result of linear stability analysis for nonlinear systems, convergence of $z_k$ to $z_\star$ is determined by the spectral radius $\rho(A_\star)$ \citep{Ostrowski1973}.
\end{proof}

The above theorem shows that the general (inexact) SQP-type method converges at a locally linear rate if $E_\star$, $C_\star$, $D_\star$ are sufficiently small.
In an exact Hessian SQP-method,  $E_\star$, $C_\star$, and $D_\star$ are zero and the rate of convergence is superlinear.
If a Hessian approximation or inexact gradients are used, the asymptotic rate of convergence is in general linear.
Thus, local convergence of such methods can be accelerated using AA as shown by Theorem~\ref{thm:AA}.

In the following, we further characterize the local convergence rate in case of a symmetric iteration matrix.

\begin{remark}[Real eigenvalues]
If $L_\star$ is symmetric, then the iteration matrix $A_\star$ has only real eigenvalues as it is similar to the symmetric matrix $K_\star^{\frac{1}{2}} A_\star K_\star^{-\frac{1}{2}}$.
This is in particular the case if exact first-order derivatives are used.
\end{remark}

\begin{theorem}(Local linear contraction rate, symmetric iteration matrix) \label{thm:rate-kappa}
Assume $C_\star$ and $D_\star$ are zero and $E_\star$ is symmetric.
Let $z^\star$ be a fixed point of the SQP-type iteration defined in \eqref{eq:iteration-map}.
Suppose SOSC, LICQ and strict complementarity hold at $z^\star$ with respect to \eqref{eq:local-QP} with $z_k = z^\star$.

Let $Z$ be a basis of the nullspace of $\tilde G(v^\star)\T$.
Furthermore, define the reduced Hessian approximation $\hat{W}_\star = Z\T W_\star Z$, as well as $\hat{E}_\star = Z\T E_\star Z$, and
\begin{align}
\hat{\Lambda}_\star = Z\T  \left(\frac{\partial}{\partial w} q(z_\star) + \frac{\partial}{\partial w} (\tilde G(v_\star) \tilde\lambda_\star)\right)Z.
\end{align}

Then, the linear asymptotic contraction rate is given by the smallest $\kappa$ satisfying
$-\kappa \hat{W}_\star \preceq \hat{E}_\star \preceq \kappa \hat{W}_\star $.
In consequence, if
\begin{align*}
    \frac{1}{1+\kappa} \hat{\Lambda}_\star \preceq \hat{W}_\star \preceq \frac{1}{1-\kappa} \hat{\Lambda}_\star
\end{align*}
holds for some $\kappa<1$, the method converges $Q$-linearly with asymptotic contraction rate $\kappa$ and a necessary condition for local convergence is given by $\hat{W}_\star \succeq \frac{1}{2} \hat{\Lambda}_\star$.
Finally, if $\hat{\Lambda}_\star \succ 0$, then $\hat{W}_\star \succ \frac{1}{2} \hat{\Lambda}_\star$ is sufficient for local linear convergence.
\end{theorem}

\begin{proof}
As exact constraint gradients are used, we have $C_\star = 0$.
Furthermore, symmetry of $\frac{\partial}{\partial w} q(v_\star)$ implies $E_\star$ is symmetric.
The result then follows directly from Theorem~4.5 in \cite{Messerer2021a}.
\end{proof}

In the following sections, we give some examples of methods proposed in the literature that fit into the considered framework.

\subsection{SQP with regularized Hessian}

If SQP is used in combination with a regularization strategy to ensure convexity of the QPs as discussed in~\cite{Verschueren2017}, the resulting SQP method is in general converging at a linear rate.
In particular, Levenberg-Marquardt regularization and the projection and mirroring techniques can be analyzed using Theorem~\ref{thm:rate-kappa}.

\subsection{(Generalized) Gauss-Newton and SCQP methods}

If the NLP \eqref{eq:NLP} takes the form
\begin{mini!}|s|
{\scriptstyle{v \in \mathbb{R}^{n_v}}}{\phi_0\left(F_0(v)\right)}{\label{eq:ggn_nlp}}{}
\addConstraint{\phi_i\left(F_i(v)\right)}{\leq 0, \quad i=1, \ldots, q \label{eq:convex_over_nonlinear}}
\addConstraint{g(v)}{= 0.}
\end{mini!}
with convex \textit{outer} functions $\phi_i$ and nonlinear \textit{inner} functions $F_i$, then the Generalized Gauss-Newton (GGN) method and Sequential Convex Quadratic Programming (SCQP) might be applied to solve the problem.
These methods use exact gradient information, but an inexact Hessian $W(z) \approx \nabla^2 \mathcal{L}(v^\star, \lambda^\star, \mu^\star)$ preserving the outer convexity of the objective function (GGN) or the outer  convexity of objective and constraints (SCQP) \citep{Verschueren2016, Messerer2021a}.
The widely used Gauss-Newton Hessian approximation is a special case of the GGN approximation and used for instance in the iLQR algorithm \citep{Li2004a, Todorov2005, Tassa2014}.
The methods converge with locally linear rate characterized by Theorem~\ref{thm:rate-kappa}.
If the GGN or SCQP Hessian approximation coincides with $\nabla^2 \mathcal{L}(v^\star, \lambda^\star, \mu^\star)$ at the solution $z^\star$, then these methods converge superlinearly and cannot be accelerated using AA.
For an in-depth discussion regarding sequential methods exploiting outer convexity, we refer to \cite{Messerer2021a}.

\subsection{Extended Gauss-Newton method}

Similarly to the GGN and SCQP method, the extended Gauss-Newton method (XGN) as introduced in \cite{Baumgaertner2022} exploits outer structure in the objective.
The outer function $\phi_0$ is in this case assumed to be a symmetric loss function and might potentially be nonconvex.
Similarly to GGN and SCQP, XGN uses exact gradient information but approximates the Hessian of the Lagrangian and thus converges with locally linear rate as characterized in Theorem~\ref{thm:rate-kappa}.

\subsection{Inexact SQP-type methods}
A prominent type of inexact SQP-type methods are Level B iterations, which use fixed constraint Jacobians, i.e., $G(v) = A$ where $A$ corresponds to the Jacobian of $g(v)$ at some linearization point $\bar v$.
They were first introduced in \cite{Bock2007}, are available in AS-RTI \citep{Nurkanovic2019a, Frey2024a} and allow one to guarantee feasibility \citep{Numerow2024, Kiessling2022}.
Recently, such inexact methods are often referred to as \textit{zero-order}, e.g. in the context of zero-order MPC \citep{Zanelli2016}, zero-order moving horizon estimation \citep{Baumgaertner2019, Baumgaertner2021}, and zero-order robust and stochastic MPC \citep{Zanelli2021, Lahr2023, Frey2024}. %
The local convergence behavior of these methods can be analyzed using Theorem~\ref{thm:rate-spectral-radius}.
An example of zero-order MPC is discussed in Section~\ref{sec:zero_order_example}.

\section{Implementation}
\label{sec:implementation}
We implemented AA in the context of SQP-type solvers in two ways.
Firstly, a general prototypical implementation with \texttt{CasADi} has been developed for the general case of arbitrary depth $m > 0$.
This implementation was used to perform initial experiments and compare different depth values.
The results of these initial experiments are presented in Section~\ref{sec:scqp_pendulum}.
Secondly, the efficient OCP solver in \texttt{acados} has been extended to perform AA with $m=1$.
This implementation, which is used in Sections~\ref{sec:furuta} and \ref{sec:zero_order_example}, has been extended with a particularly useful heuristic discussed next.

\subsection{Degraded convergence in initial regime}
Although the theory shows that AA improves local convergence, it has been observed that while the iterates are farther away from the fixed point, AA, in particular for high values of $m$ can slow down convergence.

To this end, \cite{Pollock2021} suggest deploying Anderson acceleration varying depth $m\geq0$ by heuristically dividing the space into an initial, a pre-asymptotic and an asymptotic regime, through which the depth values are increased.

In the context of SQP-type methods, our experiments showed that increasing the depth beyond $1$ does not significantly improve convergence.
Thus, we suggest the following simplified variant of Pollock's heuristic.
We add a tuning parameter $\aatresh\in\R$ and perform AA(1) whenever the current KKT residual norm is smaller than $\aatresh$; otherwise, standard iterations are performed.

\subsection{Implementation in \texttt{acados}}
In an \texttt{acados} SQP-type solver, AA can be activated by setting \texttt{with\_anderson\_acceleration = True}.
The option \texttt{anderson\_activation\_threshold} corresponds to $\aatresh$.

\section{NUMERICAL RESULTS}
\label{sec:examples}

In the following, we showcase the use of Anderson acceleration within various (inexact) SQP-type algorithms tailored to optimal control.
Code and instructions for reproducing the presented results are available at \url{https://github.com/david0oo/anderson_acceleration_sqp_type_methods}.

\subsection{SCQP of Cart Pendulum}
\label{sec:scqp_pendulum}
We consider the swing up of a cart pole system from downward pole position into an upward circular region as it was introduced in \cite{Verschueren2016}, also considered in \cite{Kiessling2025}.
The model consists of a state $x=\left[p, v, \theta, \kappa\right]^{\top}$, where $p[\mathrm{m}]$ and $v[\mathrm{m} / \mathrm{s}]$ denote the horizontal position and velocity of the cart and $\theta[\mathrm{rad}]$ and $\kappa[\mathrm{rad} / \mathrm{s}]$ are the angle and angular velocity, respectively.
The horizontal force $F[\mathrm{N}]$ is the control input.
The $x-y$ tip position of the pendulum is given by $c(x)=[p-l\sin (\theta), l \cos (\theta)]^{\top}$, where $l=0.8\;\mathrm{m}$ denotes the length of the pendulum.
We regard the OCP %
\begin{mini*}|s|
{\scriptstyle{\substack{x_0, \ldots, x_N \\
u_0, \ldots, u_{N-1}}}}
{\frac{1}{2} \sum_{k=0}^{N-1} u_k^{\top} R_k u_k}{\label{eq:cart_pendulum_scqp_ocp}}{}
\addConstraint{}{x_0=\bar{x}_0}{}
\addConstraint{}{x_{k+1}=f\left(x_k, u_k\right)\quad k=0, \ldots, N-1}{}
\addConstraint{}{\left\|\left[c(x_N) - [l, l]\T\right]\right\|_2^2-R_{\mathrm{e}}^2\leq 0,}{}
\end{mini*}
with weight $R_k=10^{-4}$, initial value $\bar{x}_0=[0,0, \pi, 0]^{\top}$ and radius for the terminal pendulum constraint $R_{\mathrm{e}}=0.05\mathrm{m}$.
The time horizon of $T=1 \mathrm{s}$ is divided into $N=20$ equidistant intervals.

In the experiment, we compare the convergence of the KKT residual for SQP with exact Hessian, GGN Hessian, SCQP Hessian and of Anderson-accelerated SQP with SCQP Hessian with memory depth $m$ of $1$, $5$, and $10$ denoted by AA$(m)$-SCQP.
The convergence of the different methods is shown in Figure~\ref{fig:cartpole_scqp_kkt_residuals}.
\begin{figure}[]
    \centering
    \includegraphics[width=.9\columnwidth]{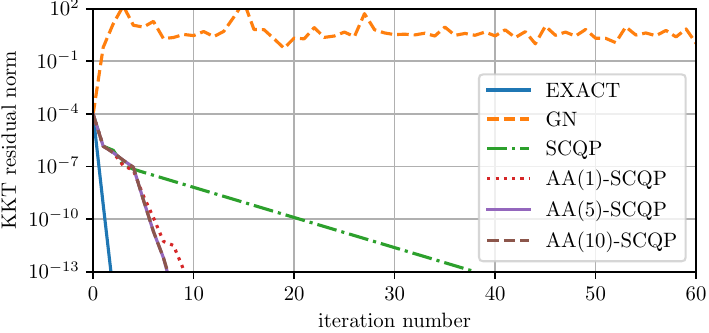}
    \caption{KKT residuals for SQP with exact Hessian, GGN, SCQP and AA-SCQP, Sec.~\ref{sec:scqp_pendulum}.}
    \label{fig:cartpole_scqp_kkt_residuals}
\end{figure}
As reported in \cite{Verschueren2016}, exact Hessian SQP exhibits quadratic local convergence, the GGN Hessian SQP does not converge, but the SCQP Hessian SQP converges linearly.
As shown in Figure~\ref{fig:cartpole_scqp_kkt_residuals}, the Anderson-accelerated SCQP converges about five times faster than SCQP.
In this example, we do not see a significant difference between the depths.

\subsection{Furuta pendulum with AA threshold}
\label{sec:furuta}
In this section, we consider a Furuta pendulum model as described in~\cite{Homburger2025a}.
We regard a modification of the OCP in~\cite{Homburger2025a}, where we added the $L_1$-cost term $10^3\abs{\theta_1 - 1.4}$ to the path cost by introducing slack variables.
When attempting to solve this OCP with an exact Hessian, we quickly run into indefiniteness and a QP solver failure.
We employ the projection regularization technique to the exact Hessian~\citep{Verschueren2017}.
This results in convergence as shown in Figure~\ref{fig:furuta_convergence}.

When only regarding the variants without a specific threshold, we observe that AA can reduce the number of iterations to solve this problem from over 60 to less than 30.
However, we can see that in the initial regime, the Anderson-accelerated variant makes less progress compared to the standard variant.
In the extreme case, when regarding a low accuracy setting, such as $0.1$ in the right plot of Fig.~\ref{fig:furuta_convergence}, as is not uncommon in the context of MPC, this can even increase the number of iterations.
Regarding the variants with different threshold values in Figure~\ref{fig:furuta_convergence}, we observe that the number of iterations can be significantly reduced with a wide range of values for $\aatresh$.
Additionally, it can be seen that the issue of slower convergence for low accuracy settings can be mitigated by choosing $\aatresh$.

\begin{figure}
    \centering
    \includegraphics[height=4.7cm]{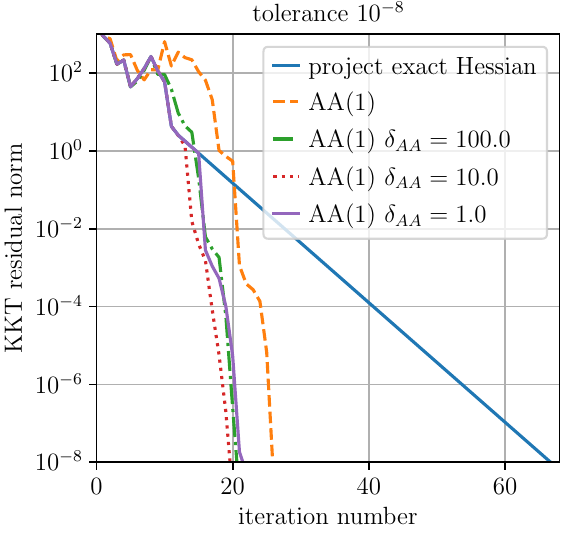}
    \includegraphics[height=4.7cm]{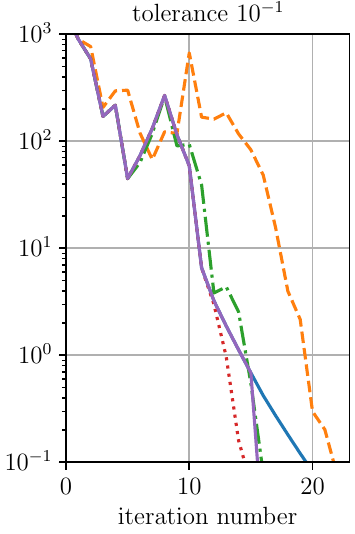}
    \caption{Convergence of exact Hessian SQP with projection regularization for different accuracies, Sec.~\ref{sec:furuta}. \label{fig:furuta_convergence}}
\end{figure}

The average computation time of one SQP iteration is $ 163 \mu \mathrm{s} $, of which the Anderson acceleration is about $ 1.6 \mu \mathrm{s} $.
This shows that the additional computational cost of Anderson acceleration with depth $m=1$ is negligible in this context.

\subsection{Zero-order OCP example}
\label{sec:zero_order_example}
We regard an OCP with a pendulum model as in Sec.~\ref{sec:scqp_pendulum}.
However, we include constraints $-80 \leq u \leq 80$, drop the terminal constraint on the pendulum position and instead use a linear least-squares cost
\begin{align*}
    \sum_{k=0}^{N-1} \frac{T}{N} x_k\T Q x_k + u_k\T R u_k + x_N\T Q x_N,
\end{align*}
with $ Q = 2\cdot\mathrm{diag}(10^3, 10^3, 10^{-2}, 10^{-2} ), R = 0.02$.
We regard a stabilizing scenario with $\bar{x}_0 = (0, 0.15\pi, 0, 0)$.
We perform zero-order iterations fixing the dynamics constraint linearizations to the values at the steady state as proposed in \cite{Zanelli2016}.
The right plot in Figure~\ref{fig:zero_order_aa} shows the solution of the zero-order scheme and the exact one.
As expected, there are significant differences.
In the left plot of Figure~\ref{fig:zero_order_aa}, we see that AA is able to significantly improve convergence of the zero-order scheme.
This experiment showcases the benefits of AA in inexact MPC schemes.

\begin{figure}
    \centering
    \includegraphics[width=.36\columnwidth]{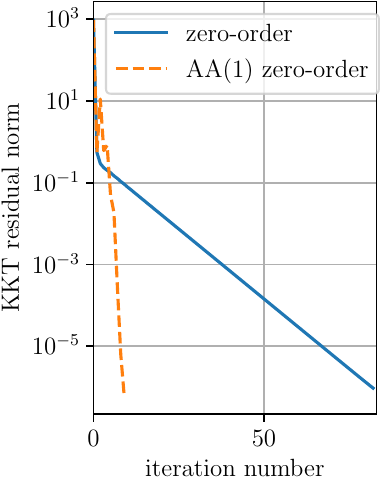}
    \includegraphics[width=.58\columnwidth]{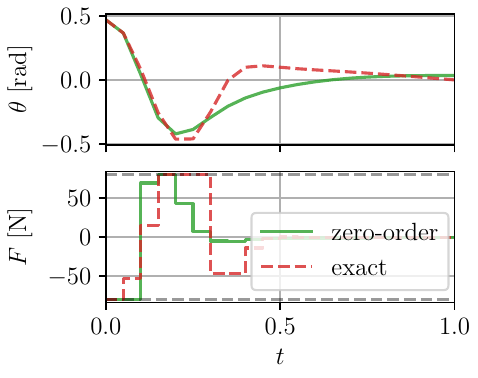}
    \caption{Left: Convergence of zero-order iterations with and without AA; Right: exact vs. zero-order solution trajectories, Sec.~\ref{sec:zero_order_example}. \label{fig:zero_order_aa}}
\end{figure}

\section{Conclusion \& Outlook}
\label{sec:conclusion}
The paper showed that Anderson acceleration (AA) can improve local convergence for a general family of SQP-type methods.
Three members of this family were considered in numerical examples from optimal control.
Since AA can deliver large speed-ups in practical MPC algorithms, the \texttt{acados} implementation is expected to be especially valuable for practitioners.
Investigating the combination of AA and real-time algorithms for MPC, such as the RTI scheme, is left for future work.